\documentclass{amsart} 
\usepackage{geometry}

\usepackage{amsmath,amssymb,amsthm}
\usepackage{subfigure}
\usepackage[english]{babel}
\usepackage{xfrac}
\usepackage{hyperref} 
\usepackage{enumerate}
\usepackage[labelfont=bf,format=plain,justification=raggedright,singlelinecheck=false]{caption} 

\newcommand{\R}{\mathbb{R}}
\newcommand{\Z}{\mathbb{Z}}
\newcommand{\N}{\mathbb{N}}
\newcommand{\F}{\mathbb{F}_2}

\newcommand{\HFK}{\mathit{HFK}}
\newcommand{\CFK}{\mathit{CFK}}
\newcommand{\HF}{\mathit{HF}}
\newcommand{\CF}{\mathit{CF}}
\DeclareMathOperator{\Ord}{Ord}

\DeclareMathOperator{\br}{br}
\DeclareMathOperator{\gr}{gr}

\newcommand{\abs}[1]{\lvert#1\rvert}

\newtheoremstyle{theorem}{}{}{\itshape}{}{\bfseries}{.}{ }{} 
\newtheoremstyle{definition}{}{}{}{}{\bfseries}{.}{ }{} 

\theoremstyle{theorem}
\newtheorem{Theorem}{Theorem}[section]
\newtheorem{theorem}[Theorem]{Theorem}
\newtheorem{lemma}[Theorem]{Lemma}
\newtheorem{proposition}[Theorem]{Proposition}
\newtheorem{corollary}[Theorem]{Corollary}
\newtheorem{construction}[Theorem]{Construction}
\newtheorem{question}[Theorem]{Question}

\numberwithin{equation}{section}

\theoremstyle{definition}
\newtheorem{definition}[Theorem]{Definition}
\newtheorem{rem}[Theorem]{Remark}

\numberwithin{equation}{section}

\numberwithin{equation}{section}

\begin{document}

\title[Multiplicity of Knot Floer order for cables]{On the multiplicity of Knot Floer order under cabling}
\author{David Suchodoll}
\address{Humboldt-Universit\"at zu Berlin, Rudower Chaussee 25, 12489 Berlin, Germany.}
\email{suchodod@hu-berlin.de}


\date{\today} 

\begin{abstract}
The \emph{knot Floer order} $\Ord(K)$ is a knot invariant derived from knot Floer homology that provides bounds on many other invariants, such as the bridge index $\br(K)$ for which $\Ord(K) + 1 \leq \br(K)$. For all $(p,q)$-cables of L-space knots, we show that $\Ord(K) + 1$ is multiplicative in $p$ when $g(K) > 1$, and the same holds for $g(K) = 1$ provided $q > 2p$. We also compute the knot Floer order in the range $q < 2p$, thereby determining $\Ord(K_{p,q})$ in terms of $\Ord(K)$ for all cables of L-space knots. We establish upper bounds under cabling for $\Ord(K)$ and discuss potential applications to a conjecture by Krishna and Morton, proving that the braid index of an L-space cable appears as an exponent in its Alexander polynomial if it does for its companion, provided $\Ord(K)+1$ is multiplicative.
\end{abstract}

\keywords{Knot invariants, knot Floer Homology, L-space knots, torsion order} 

\makeatletter
\@namedef{subjclassname@2020}{%
  \textup{2020} Mathematics Subject Classification}
\makeatother

\subjclass[2020]{57K10; 57K14, 57K16, 57K18} 

\maketitle



\section{Introduction}
The \emph{knot Floer order} $\Ord(K)$ of a knot $K$ in $S^3$ was introduced in~\cite{JuhMilZem2020cobordism} as the torsion order of the $\F[U]$-module $\HFK^-(K)$. It offers new lower bounds for knot invariants that were previously difficult to estimate, such as the \emph{fusion number} for ribbon knots (see e.g.~\cite{KaMa1997Ribbon, HoKaPa2021Ribbon}), \emph{band unlinking number} (see e.g.~\cite{Lic1986bandunlinking, HoYa1990bandunlinking}), \textit{unknotting number} (see e.g.~\cite{AkrEam2020unknotting,Zemke2019Linkcobordism}), and \emph{bridge index} $\br(K)$ (see e.g.~\cite{Schubert1954Bruecken, schultens2001additivitybridgenumberknots}). Among these, the bridge index is particularly relevant to the present work. Juhász, Miller, and Zemke~\cite{JuhMilZem2020cobordism} established that 
\begin{equation}
\label{eq: Ord+1 < br}
    \Ord(K) + 1 \leq \br(K).
\end{equation}
We prove that the lower bound in Equation~\eqref{eq: Ord+1 < br} is multiplicative under cabling of L-space knots. For $p$ and $q$ coprime, the $(p,q)$-cable of a knot $K$, denoted $K_{p,q}$, refers to the satellite knot with companion $K$ and pattern the torus knot $T_{p,q}$. This is a similar result to~\cite[Proposition 1.4]{hom2022unknottingnumbercabling}, which asserts the lower bound $\Ord(K_{p,q})+1 \geq p(\Ord(K) + 1)$. 

\begin{theorem}
\label{thm: Ord_v is multiplicative}
    Let \(K_{p,q}\) be the $(p,q)$-cable of an L-space knot \(K\) in \(S^3\) with Seifert genus $g(K) > 1$. Then 
    \[
    \Ord(K_{p,q}) + 1 = p \cdot (\Ord(K) + 1).
    \]
\end{theorem}

A knot is called an \emph{L-space knot} if it admits a positive Dehn surgery to an \emph{L-space} -- a 3-manifold with the "simplest" Heegaard Floer homology. L-space knots were introduced in~\cite{OS2005lensespacesurgeries} and have notable properties, e.g. being fibered~\cite{Ni2007fibered}. Among the L-space knots, the trefoil knot \(T_{2,3}\) stands out as the unique genus-one knot\footnote{It is one of two fibered genus-one knots~\cite[5.14 Proposition]{BurZieHeu2014knots}. However, the knot $K4a1$ is not an L-space knot, which can be seen from its Alexander polynomial. }. The analogue of Theorem~\ref{thm: Ord_v is multiplicative} for \(T_{2,3}\) holds under certain conditions.

\begin{theorem}
    \label{thm: mult for T23}
    Let \(K\) denote the trefoil knot \(T_{2,3}\), and let \(K_{p,q}\) be its \((p,q)\)-cable, where \(p\) and \(q\) are coprime integers. 
    \begin{align*}
        \Ord(K_{p,q}) + 1 = \left\{
        \begin{array}{cc}
            p + 1 & q < p \\
            q & p < q < 2p \\
            2p & 2p < q
        \end{array}
        \right.
    \end{align*}
\end{theorem}

Our results have particular relevance to the calculation of the bridge index and thus to the \emph{BB-Conjecture}~\cite{KriMor2025bbconjecture}. It states that the bridge index $\br(K)$ and the \emph{braid index} $\operatorname{b}(K)$ agree if $K$ is an L-space knot. Both knot invariants are known to be multiplicative under cabling~\cite{williams1992multiplicative}, and 
\begin{equation}
\label{eq: br < braid}
    \br(K) \leq b(K), 
\end{equation}
but both are generally difficult to calculate.

\ifx
Assuming $p \cdot (\Ord(K_{p,q}) + 1) > \Ord(K) + 1$ yields that $p\cdot\br(K) = \br(K_{p,q}) \geq \Ord(K_{p,q}) + 1 > p \cdot ( \Ord(K) + 1)$ thus giving a higher bound on $\br(K)$ then Equation~\eqref{ineq: ord and br}. We prove in Theorem~\ref{thm: Ord_v is multiplicative} that this is not the case. 
\fi

For L-space knots, the following results were proven. Juh\'asz, Miller, and Zemke prove in~\cite[Corollary~5.3]{JuhMilZem2020cobordism} that 
\begin{align}
\label{eq: Ord detects br}
\br(K) = \Ord(K) + 1
\end{align}
for torus knots $T_{p,q}$. Krishna and Morton~\cite[Equation~3.2, p.7]{KriMor2025bbconjecture} prove Equation~\eqref{eq: Ord detects br} for L-space knots which are closures of positive braids that include a full twist, and our data in~\cite{census_knot_paper} provides only 9 possible counterexamples to Equation~\eqref{eq: Ord detects br}. Thus we ask

\begin{question}
\label{qu: Ord+1 = br}
    Let $K$ be an L-space knot in $S^3$. For which $K$ does $\Ord(K) + 1$ equal the bridge index $\br(K)$? 
\end{question}
There are counterexamples to the Equality~\eqref{eq: Ord detects br} for L-space knots coming from $(p,q)$-cables of $T_{2,3}$ where $p < q < 2p$. The $(2,3)$-cable was identified in~\cite{hom2022unknottingnumbercabling}. The bridge index satisfies $\br(K) = p$~\cite{williams1992multiplicative}. However, $\Ord(T_{2,3:p,q}) + 1$ is given by Theorem~\ref{thm: mult for T23}. For closures of positive braids that include a full twist, Himeno~\cite{himeno2025bridgeindexbraidindex} confirms Question~\ref{qu: Ord+1 = br}, even if the knot is not L-space.

In the case of L-space knots, the complex $\CFK(K)$ is determined by the Alexander polynomial $\Delta_K(t)$~\cite{OS2005lensespacesurgeries}. Juh\'asz, Miller, and Zemke~\cite[Lemma~5.1]{JuhMilZem2020cobordism} prove that $\Delta_K(t)$ detects $\Ord(K)$ as the highest difference of neighboring exponents.
\begin{equation}
\label{eq: ord from alex}
    \Ord(K) = \max\{ \alpha_{i-1} - \alpha_i \mid i = 0,\dots, 2n\},
\end{equation}
where $\alpha_i$ are the exponents in $\Delta_K(t)$ in decreasing order. Proving Theorem~\ref{thm: Ord_v is multiplicative} using the known cabling formula for the Alexander polynomial (see e.g.~\cite[Theorem~6.15]{Lickorish1997introduction})
\[
\Delta_{K_{p,q}}(t) = \Delta_K(t^p) \cdot \Delta_{T_{p,q}}(t)
\]
proved to be very difficult, even restricted to the case of L-space cables. It is due to the ease of a new tool -- immersed curves, introduced in~\cite{HanRasWat2022heegaardfortorusbdd, HanRasWat2024borderedviaimmersed} -- that we did not pursue a proof using the Alexander polynomial. The immersed curve invariant $\widehat{\HF}(M)$ consists of curves in a punctured torus and encodes a certain bordered invariant. In the case of knot exteriors $M = S^3 \setminus \operatorname{int}(\nu K)$, it decodes $\HFK^-(K)$ and thus $\Ord(K)$. It behaves nicely under cabling~\cite{HanWat2023cabling} and for L-space knots~\cite[Section 7.5]{HanRasWat2024borderedviaimmersed}. The proof of Theorem~\ref{thm: Ord_v is multiplicative} is an explicit geometric argument using the immersed curves~\cite[Proposition~3.1]{hom2022unknottingnumbercabling}. However, we want to encourage any reader to prove Theorem~\ref{thm: Ord_v is multiplicative} using $\Delta_K(t)$ or possibly extend the structural results on the Alexander polynomial of L-space knots (see e.g.~\cite{Krcatovich2018restriction}), possibly answering Problem~33 in~\cite{HomLipRub201730years}, stating the question of what polynomials occur as Alexander polynomials of L-space knots.
\subsection{Multiplicity for cables of general knots}
Krishna and Morton~\cite{KriMor2025bbconjecture} proved that if the BB-conjecture holds for an L-space knot, it also holds for all of its L-space cables. We prove that the multiplicity condition exhibits a similar pattern.

\begin{theorem}
\label{q: mult for cables}
Let $K$ be a knot in $S^3$ with $\varepsilon(K) \geq 0$. 
Suppose there exists some pair $(p,q)$ such that
\[
\Ord(K_{p,q}) + 1 = p \cdot (\Ord(K) + 1).
\]
Then the same multiplicity formula holds for all $(p,q)$-cables of $K$ with
\[
q > 2p \abs{\tau(K)}
\]
\end{theorem}
Here, $\tau(K)$ denotes the Ozsv\'ath--Szab\'o concordance invariant introduced in~\cite{OS2003fourballgenus} and $\varepsilon(K)$ Hom's concordance invariant introduced in~\cite{Hom2014concordance}. 

To prove Theorem~\ref{q: mult for cables}, we analyze each arc in the immersed curve, strengthening~\cite[Proposition~3.1]{hom2022unknottingnumbercabling} by Proposition~\ref{rem: main advancement}, thus proving the following explicit bounds.

\begin{proposition}
    
\label{thm: explicit bounds}
Let $K$ be a knot in $S^3$ and $K_{p,q}$ its $(p,q)$-cable. Then 
\[
\Ord(K)+1 \geq p^{-1}\Ord(K_{p,q}) \geq \Ord(K) - 1.
\]
\end{proposition}

\subsection{Alexander polynomial of iterated cables of L-space knots}

An application of our results is provided by Conjecture~1.9 of~\cite{KriMor2025bbconjecture}, which addresses hyperbolic L-space knots. It conjectures that for a hyperbolic L-space knot $K$ such that $b(K) = b$ the Alexander polynomial has the form 

\begin{equation}
\label{eq: Alex of L-space det bridge}
    \Delta_K(t) = 1 - t + t^b - \cdots.
\end{equation}

We can prove the following.

\begin{theorem}
\label{thm: L-space Alex bridge}
    Let $K$ be an L-space knot with braid index $b(K) = b$. Let $K_{p,q}$ be the L-space $(p,q)$-cable of $K$ with braid index $b(K_{p,q}) = p\cdot b(K)$. If $\Ord(K_{p,q}) + 1 = p(\Ord(K) + 1)$, then 
    \begin{equation*}
        \Delta_{K_{p,q}}(t) = 1 - t + t^{p\cdot b(K)} - \cdots.
    \end{equation*}
\end{theorem}

Notably, the proof of Theorem~\ref{thm: L-space Alex bridge} relies on the multiplicity $\Ord(K) + 1$ under cabling. As a result, certain cables of $T_{2,3}$ are excluded. This restriction was already observed in~\cite{KriMor2025bbconjecture}, which identifies the $(2,3)$-cable of $T_{2,3}$ as a counterexample to Equation~\eqref{eq: Alex of L-space det bridge}. It is an L-space knot, but it does not satisfy the inequality $2p < q$ from Theorem~\ref{thm: mult for T23}.

This paper is a follow-up to~\cite{census_knot_paper}, which is in preparation. Theorem~\ref{thm: Ord_v is multiplicative} was conjectured based on data obtained by studying the bridge index of the census knots in that ongoing project~\cite{census_knot_paper}. The explicit computations leading to the main theorem and both questions are publicly available in~\cite{GitHub}. Both computational projects rely partly on data from~\cite{knotinfo}.  

The paper is organized as follows. Section~\ref{sec: prereq} reviews the background on knot Floer homology and the immersed curve invariant $\widehat{\HF}(M)$. Section~\ref{sec: proof of theorem} recalls key results from~\cite{HanRasWat2022heegaardfortorusbdd} and proves Theorems~\ref{thm: Ord_v is multiplicative} and~\ref{thm: mult for T23}. Section~\ref{sec: multiplicity in general} extends these results to cables of arbitrary knots.

\subsection*{Acknowledgments}
I thank Klaus Mohnke for opportunities to present my work and for his insightful questions. Special thanks to Marc Kegel for the census knots project, whose data inspired the conjecture leading to the main theorem, and for his careful reading of several drafts. I am grateful to Claudius Zibrowius for suggesting the use of the immersed curve invariant and providing helpful comments, and to Andr\'as Juh\'asz, Maggie Miller, and Ian Zemke for clarifying their results. Thanks to Naageswaran Manikandan and Chun-Sheng Hsueh for spotting errors and proofreading, and to Charles Livingston and Allison H. Moore for sharing extensive data. Finally, I thank Tye Lidman, Jen Hom, and JungHwan Park for correspondence that resolved questions in an earlier draft now appearing as theorems.

\section{Preliminaries}
\label{sec: prereq}

\subsection{Knot Floer homology}
\label{sec:knot Floer homology}
The filtered chain homotopy type of the knot Floer complex $\CFK(K)$ is an invariant of a knot $K$ in $S^3$~\cite[Theorem~3.1]{OS2004knotinvariants}.
The complex is freely generated over $\F[U,V]/(UV)$. The differential counts pseudoholomorphic disks $\phi$, where $n_z(\phi)$ and $n_w(\phi)$ count the disks' intersection with certain submanifolds corresponding to basepoints $z$ and $w$. We denote the homotopy classes of these disks from $x$ to $y$ by $\pi_2(x,y)$. The expected dimension of the moduli space $\mathcal{M}(\phi)$ of pseudoholomorphic representatives of $\phi$ is denoted by $\mu(\phi)$. The differential is given by 
\begin{align*}
\partial x = \sum_{y} \sum_{\substack{\phi \in \pi_2(x,y)\\ \mu(\phi) = 1}} \# \left(\mathcal{M}(\phi)/\R\right) ~ U^{n_w(\phi)}V^{n_z(\phi)}\cdot y.
\end{align*}
The Alexander grading is defined by $\gr_A(x) - \gr_A(y) = n_z(\phi) - n_w(\phi)$ for $\phi \in \pi_2(x,y)$. It is fixed by the quasi-isomorphism to $\CF(S^3) \cong \F[U,V]/(UV)$.

We denote the $V=0$ specialization by $\HFK^-(K)$. It can be decomposed as  
\begin{align}
\label{eq: HFK^- decomposition}
    \HFK^-(K) = \F[U] \oplus \HFK^-_{\operatorname{red}}(K) \quad\text{where}\quad
\HFK^-_{\operatorname{red}}(K) = \bigoplus_{k = 0}^n \F[U]/U^{i_k}.
\end{align}

\begin{definition}[\cite{JuhMilZem2020cobordism}]
    Let $K$ be a knot in $S^3$. Let $\HFK^-_{\operatorname{red}}(K)$ be the torsion submodule of its knot Floer homology. The \emph{knot Floer order} is defined as
    \[
    \Ord(K) = \min\{k \in \N \mid U^k \cdot \HFK^-_{\operatorname{red}}(K) = 0\}.
    \]
\end{definition}

Since the decomposition in~\eqref{eq: HFK^- decomposition} is finite, $\Ord(K)$ is always finite. Details on the composition can be found in~\cite[Proposition~6.1.4, Proposition~7.3.3]{OSS2015gridhomology}. 

\subsection{Immersed curves}
\label{sec:immersed curves}

The \emph{immersed curve invariant} \( \widehat{\HF}(M) \) for a 3-manifold \( M \) with torus boundary was introduced by Hanselman, Rasmussen, and Watson in~\cite{HanRasWat2022heegaardfortorusbdd, HanRasWat2024borderedviaimmersed}. This invariant is defined as a collection  
\[
\boldsymbol{\gamma} = \{\gamma_0, \dots, \gamma_n\}
\]  
of curves \( \gamma_i: S^1 \looparrowright T^2 \setminus \{z\} \), up to regular homotopy, where each curve is an immersion into the punctured torus. In~\cite{HanRasWat2024borderedviaimmersed}, the authors prove that \( \widehat{\HF}(M) \) encodes the same data as a certain type of bordered Floer invariant, the type D structure \( \widehat{\operatorname{CFD}}(M) \) introduced in~\cite{LipOszThu2018bordered}. In the special case where \( M \) is the exterior of a knot \( K\) in \(S^3 \), i.e. \( M = S^3 \setminus \operatorname{int}(\nu K) \), the invariant is equivalent to the knot Floer complex \( \CFK(K) \)~\cite[Theorem 2]{KWZ2023amnemonic}.

Let \( M = S^3 \setminus \operatorname{int}(\nu K)\), where \( \nu K \) denotes a tubular neighborhood of the knot \( K\) in \(S^3 \). The boundary torus is identified with \( \partial M\). We denote by \( \mu \) a meridian and by \( \lambda \) a Seifert longitude of the knot exterior. The basepoint in \( T^2 \) is taken to be the point of transverse intersection \( z = \lambda \cap \mu \). The curve $\gamma_0 \in \boldsymbol{\gamma}$ is the unique component connected to $\{\pm\frac12\} \times \R$. It is a concordance invariant~\cite[Proposition~2]{HanWat2023cabling}. 

Let \( \R^2 \) be the universal cover of the torus, with the lifts of the basepoint \( z \) lying at the lattice points \( \Z^2 + (0, 1/2) \). We lift the set of curves \( \boldsymbol{\gamma} \) to the intermediate covering space \( \overline{T} = \R^2 / \lambda \), where \( \lambda \) is chosen to run in the horizontal direction. View \( \overline{T} \) as the infinite strip \( [-1/2, 1/2] \times \R \), with the identification \( (-1/2, t) \sim (1/2, t), \forall t \in \R \). We refer to lifts of the basepoint \( z \) as \emph{pegs}. The intersections of $\widetilde{\gamma} \cap \{0\}\times \R$ will be homotoped to appear at integer heights. See Figure~\ref{fig:T23 immersed curve}. The lift of $\gamma_0$ encodes the Ozsv\'ath-Szab\'o $\tau$-invariant as its highest intersection with $\mu$ and Hom's $\varepsilon$-invariant by its behavior after the intersection. Turning down, up, or straight corresponds to the values $+1, -1$, or $0$~\cite[Proposition~2]{HanWat2023cabling}.

\begin{figure}[htbp]
    \centering
    \subfigure[$\tau=1$, $\varepsilon = 1$]{%
        \includegraphics[width=0.15\textwidth]{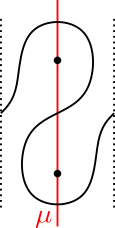}
    }
    \hspace{1cm}
    \subfigure[$\tau=0$, $\varepsilon = 0$]{%
        \includegraphics[width=0.15\textwidth]{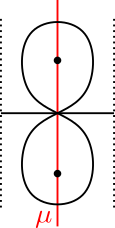}
    }
    \hspace{1cm}
    \subfigure[$\tau=-1$, $\varepsilon = -1$]{%
        \includegraphics[width=0.15\textwidth]{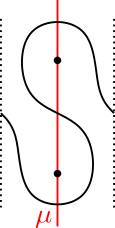}
    }
    \caption{The smoothed immersed curves $\boldsymbol{\gamma}$ for~$(a)$ $T_{2,3}$,~$(b)$ $K4a1$, and~$(c)$ $T_{-2,3}$.}
    \label{fig:T23 immersed curve}
\end{figure}

To obtain $\Ord(K)$ from the immersed curve, we refer to a result by Hom--Lidman--Park~\cite{hom2022unknottingnumbercabling}. Let $(\alpha_i)_{i = 1}^{2n}$ be the arcs obtained by splitting the immersed curve at the transverse intersections with $\mu = \{0\} \times \R$. Their \textit{length} is defined by the number of pegs between the two intersections $\alpha \cap \mu$. An arc is called \textit{right arc} if it is contained in $[0, 1/2) \times \R$.

\begin{lemma}[Lemma~2.7~\cite{hom2022unknottingnumbercabling}]
\label{lem: homlidpa}
    Let $\HFK^-(K) = \F[U] \oplus \F[U]/U^d_i$. Then for each $d_i$, the immersed curve associated to $\HFK^-(K)$ has a right arc of length $d_i$. Conversely, an immersed curve with a right arc of length $d_i$ comes from a complex $\CFK^-(K)$ with an $\F[U]/U^{d_i}$ summand in $\HFK^-$. \qed
\end{lemma}
In particular, $\Ord(K) = \max\{\operatorname{length}(\alpha_i)\mid i = 1, \dots, n\}$.
\begin{rem}
    In general, the complex generated by the intersections $\widetilde{\gamma} \cap \mu$ with grading $\gr_A(x) = \operatorname{height}(x)$. The differential is given by counting bigons bounded by $\mu$ and $\widetilde{\gamma}$ (cf.~\cite[Section~4]{HanRasWat2022heegaardfortorusbdd}).
\end{rem}

To prove our results, we analyze each type of arc using Proposition~3.1 of~\cite{hom2022unknottingnumbercabling}. 
We follow their exposition; by exploiting $\Ord(K) = \Ord(-K)$~\cite[Section~5.1]{OzsSza2006Holomorphictriagles}, and
\(\varepsilon(-K) = -\varepsilon(K)\)~\cite[Proposition~3.6]{Hom2014tauofcables}, we may assume without loss of generality that 
\(\varepsilon(K) \geq 0\). This reduces the problem to the study of right arcs 
(see Section~2 of~\cite{hom2022unknottingnumbercabling} for a concise overview, 
and~\cite{hanselman2023knotfloerhomologyimmersed} for a full exposition of immersed 
curves for knotlike complexes). 

Let $K_{p,q}$ be the $(p,q)$-cable of a knot $K$ in $S^3$. To obtain the immersed curve $\gamma_{p,q}$ for $K_{p,q}$ from the immersed curve $\gamma$ of $K$, we recall the main theorem from~\cite{HanWat2023cabling}.
\begin{construction}[Theorem~1~\cite{HanWat2023cabling}]
\label{prop: cabling process}
    Let $\boldsymbol{\gamma}$ be the immersed multicurve associated with $K$ and $\gamma_{p,q}$ the multicurve associated with the cable $K_{p,q}$. Let $\widetilde{\gamma}$ and $\widetilde{\gamma}_{p,q}$ be their lifts to $\R^2 \setminus \Z^2$. Then $\widetilde{\gamma}_{p,q}$ is the curve obtained by the following process. 
    \begin{enumerate}
    \item Draw $p$ copies of $\widehat{HF}(M)$ next to each other, each scaled vertically by a factor of $p$, staggered in height such that each copy of the curve is $q$ units lower than the previous copy.
    \item Connect the loose ends of the successive copies of the curve.
    \item Translate the pegs horizontally so that they lie in the same vertical line, carrying the curve along with them.
\qed
\end{enumerate}
\end{construction}

Following~\cite{hom2022unknottingnumbercabling}, we say that an arc is \emph{essential} 
if it connects to \(\{\pm \tfrac12\} \times \mathbb{R}\), and \emph{initial} if it 
contains the intersection of \(\mu\) with an essential arc. 
Every noninitial right arc in the immersed curve between the highest and lowest 
intersection points falls into exactly one of four types, illustrated in 
Figure~\ref{fig: eta pm}. We follow the naming scheme of~\cite{HanWat2023cabling}.

\begin{figure}[htbp]
    \centering
    \subfigure[$\eta_n^{--}$]{%
        \includegraphics[height=1.8in]{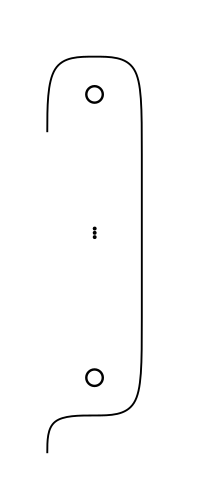}
    }
    \hspace{1cm}
    \subfigure[$\eta_n^{-+}$]{%
        \includegraphics[height=1.8in]{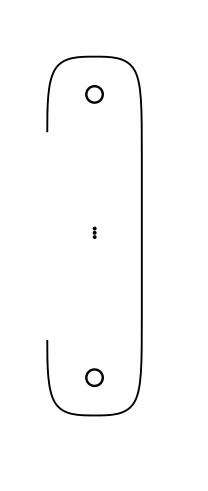}
    }
    \hspace{1cm}
    \subfigure[$\eta_n^{+-}$]{%
        \includegraphics[height=1.8in]{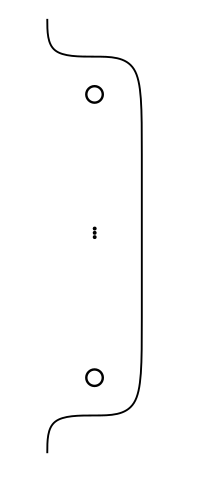}
    }
    \hspace{1cm}
    \subfigure[$\eta_n^{++}$]{%
        \includegraphics[height=1.8in]{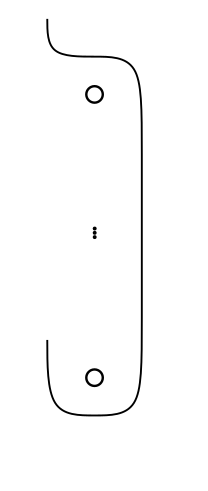}
    }
    \caption{Classification of arcs \(\eta_n^{\pm\pm}\) in the immersed curve. 
The first sign indicates the orientation at the top endpoint, and the second sign the orientation at the bottom endpoint. 
The subscript \(n\) denotes the arc length.}
\label{fig: eta pm}
\end{figure}

\begin{proposition}[{\cite[Proposition~3.1(1)]{hom2022unknottingnumbercabling}}]
\label{prop: noninitial arcs}
Let $\eta^{\pm \pm}_n$ be a noninitial right arc in the immersed curve. 
Under $(p,q)$–cabling, the following hold:
\begin{enumerate}[(1)]
    \item A right arc $\eta^{--}_n$ yields a right arc of length $pn$.
    \item A right arc $\eta^{-+}_n$ yields a right arc of length $pn - p + 1$.
    \item A right arc $\eta^{+-}_n$ yields a right arc of length $pn + p - 1$.
    \item A right arc $\eta^{++}_n$ yields a right arc of length $pn$. 
\qed
\end{enumerate}
\end{proposition}

\begin{rem}
If the arc of maximal length is of type~(1) or~(4), both before and after 
$(p,q)$–cabling, then $\Ord(K)$ is multiplicative. 
If in both instances the arc of maximal length is of type~(2), then $\Ord(K) - 1$ is multiplicative. Type (3) is most important for Theorem~\ref{thm: Ord_v is multiplicative}. If the maximal length arc is of type (3), we get that \[\Ord(K_{p,q}) + 1 = pn + p = p(n + 1) = p(\Ord(K) + 1).\]
\end{rem}

To complete the analysis of $\gamma_{p,q}$, it remains to consider \emph{initial} right arcs. These can appear in any form from Figure~\ref{fig: eta pm}, but they give rise to five additional cases.

\begin{figure}[htbp]
    \centering
    \subfigure[$\eta_n^{00}$]{%
        \includegraphics[height=1.8in]{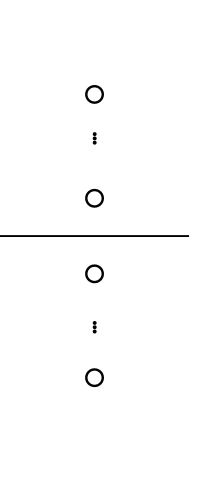}
    }
    \hspace{0.5cm}
    \subfigure[$\eta_n^{-0}$]{%
        \includegraphics[height=1.8in]{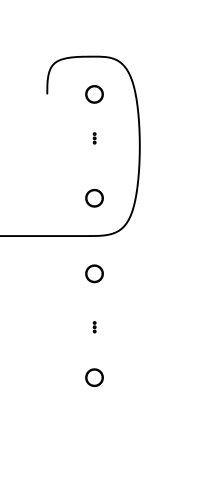}
    }
    \hspace{0.5cm}
    \subfigure[$\eta_n^{+0}$]{%
        \includegraphics[height=1.8in]{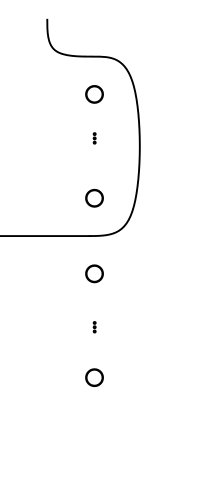}
    }
    \hspace{0.5cm}
    \subfigure[$\eta_n^{0-}$]{%
        \includegraphics[height=1.8in]{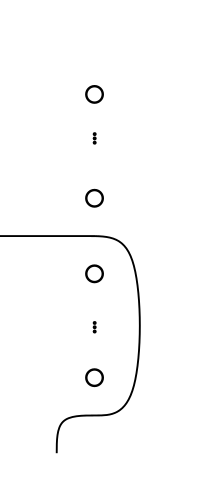}
    }
    \hspace{0.5cm}
    \subfigure[$\eta_n^{0+}$]{%
        \includegraphics[height=1.8in]{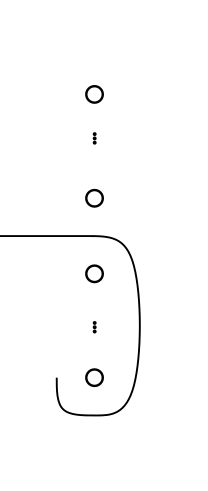}
    }
    \caption{Initial arcs in the immersed curve, denoted by $\eta_n^{0\pm}$ or $\eta_n^{\pm 0}$.
The first superscript sign describes the arc’s behavior at the top, the second at the bottom, with $0$ indicating a connection to the essential arc.
The subscript $n$ records the arc’s length.}
    \label{fig: eta pm0}
\end{figure}

\begin{proposition}[{\cite[Proposition~3.1(2)]{hom2022unknottingnumbercabling}}]
\label{prop: initial arcs}
Suppose $\varepsilon(K) = 1$, and let $\eta^{\pm \pm}_n$ or $\eta^{0 \pm}_n$ be an \emph{initial} right arc. 
Under $(p,q)$–cabling, the following hold:

\begin{enumerate}[(1)]
    \item Right arcs $\eta^{--}_n$, $\eta^{+-}_n$, and $\eta^{0-}_n$ yield right arcs of length
    \begin{itemize}
        \item $pn$ for $q < p(2\tau - 1)$,
        \item $pn + p - 2p\tau + q - 1$ for $p(2\tau - 1) < q < 2p\tau$,
        \item $pn + p - 1$ for $q > 2p\tau$.
    \end{itemize}
    \item Right arcs $\eta^{-+}_n$, $\eta^{++}_n$, and $\eta^{0+}_n$ yield right arcs of length
    \begin{itemize}
        \item $pn - p + 1$ for $q < p(2\tau - 1)$,
        \item $pn - 2p\tau + q$ for $p(2\tau - 1) < q < 2p\tau$,
        \item $pn$ for $q > 2p\tau$.
    \end{itemize}
\end{enumerate}

In each case, the resulting right arc is noninitial, and the second superscript remains unchanged. \qed
\end{proposition}

\begin{proposition}
\label{rem: main advancement}
The collections of right arcs described in Propositions~\ref{prop: noninitial arcs} and~\ref{prop: initial arcs}, while not exhaustive, contain all arcs of maximal length in the cable.
\end{proposition}

\begin{proof}
Let $\boldsymbol{\gamma}$ be the immersed multicurve before cabling with a maximal arc of length $n$. By Proposition~3.1 of~\cite{hom2022unknottingnumbercabling}, the $(p,q)$-cabling produces an arc of the corresponding type and length in the \textit{rightmost} copy of $\gamma_{p,q}$ (cf. Construction~\ref{prop: cabling process})).

It remains to analyze arcs outside the rightmost copy. Since the pegs are spaced at intervals of $p$, this forces any such arc to be broken at a length of at most $p$. On the other hand, Proposition~3.2 of~\cite{hom2022unknottingnumbercabling} ensures the existence of an arc of length $p$ in the rightmost copy. Therefore, no arc outside the rightmost copy can realize maximal length.
\end{proof}

\section{Proof of theorems}
\label{sec: proof of theorem}

\subsection{Multiplicity for cables of L-space knots}

L-space knots have been defined initially, and studied by Ozsváth and Szabó in~\cite{OS2005lensespacesurgeries}. Their Alexander polynomial is alternating with coefficients $\pm 1$ and determines the knot complex completely, as proven in Theorem~1.2 and Corollary~3.1 in~\cite{OS2005lensespacesurgeries}. Most importantly, their immersed curves are monotone in a neighborhood of $\mu$~\cite[Section~7.5]{HanRasWat2024borderedviaimmersed}. This suffices to prove Theorem~\ref{thm: Ord_v is multiplicative}.

\begin{proof}[Proof of Theorem~\ref{thm: Ord_v is multiplicative}]
Since the immersed curve is monotonically decreasing in a neighborhood of $\mu$, all noninitial arcs are of type $\eta_n^{+-}$, while the initial arc is of type $\eta_1^{--}$~\cite{Krcatovich2018restriction}. We can assume that after homotopy, there is no arc of length zero. 

We distinguish two cases for the arc of maximal length:

\begin{enumerate}
    \item \textbf{Noninitial maximal arc:} Suppose the arc of maximal length is noninitial and thus of type $\eta_n^{+-}$. Then, by Proposition~\ref{prop: noninitial arcs}, the corresponding cabled curve $\gamma_{p,q}$ contains an arc of type $\eta^{+-}_{pn + p - 1}$. Therefore,
    \[
    \Ord(K_{p,q}) + 1 \ge pn + p = p(\Ord(K) + 1).
    \]
    
    \item \textbf{Initial maximal arc:} Suppose the maximal arc is initial and thus of type $\eta_1^{--}$. Since $g(K) > 1$, the diagram contains at least two right arcs, including a noninitial arc of type $\eta_1^{+-}$. The noninitial arc has length at least one, reducing this case to the previous one.
\end{enumerate}

Finally, Proposition~\ref{prop: initial arcs} ensures that any initial arc has length at most $2p - 1$, which is no larger than the arcs coming from $\eta_n^{+-}$ arcs, and Proposition~\ref{rem: main advancement} shows that there is no longer arc.
\end{proof}

\begin{proof}[Proof of Theorem~\ref{thm: mult for T23}]
The immersed curve of $T_{2,3}$ contains a single initial arc of type $\eta_1^{--}$. Using $\tau(T_{2,3}) = 1$ and $n = 1$, Proposition~\ref{prop: initial arcs}(1) immediately gives the result, since it yields the existence of arcs of length
\begin{itemize}
        \item $p$ for $q < p$,
        \item $q - 1$ for $p < q < 2p$,
        \item $2p - 1$ for $q > 2p$.
\end{itemize}
By Proposition~\ref{rem: main advancement}, there cannot be a longer arc in the immersed curve of the cable. 
\end{proof}

\begin{rem}
    These proofs are mere applications of the result by Hom--Lidman--Park. The only advancement lies in the realization that any right arc not in the rightmost copy is at most of length $p$, as stated in Proposition~\ref{rem: main advancement}.
\end{rem}

Notably, we can prove Theorem~\ref{q: mult for cables} and Proposition~\ref{thm: explicit bounds} using the same methods. 

\subsection{Multiplicity for general knots}
\label{sec: multiplicity in general}

\begin{proof}[Proof of Theorem~\ref{q: mult for cables}]
Let $K_{p,q}$ be the $(p,q)$-cable of a knot $K$ in $S^3$ such that 
\[
\Ord(K_{p,q}) + 1 = p \cdot (\Ord(K) + 1).
\]
Since $\Ord(K)$ is realized by the maximal length of a right arc in the immersed curve (Lemma~\ref{lem: homlidpa}), the multiplicity of $\Ord(K)+1$ implies that the arc of maximal length $\alpha$ is either noninitial of type $\eta^{+-}$, in which case there is nothing more to prove (cf. proof of Theorem~\ref{thm: Ord_v is multiplicative}), or initial of type $\eta^{--}$, $\eta^{+-}$, or $\eta^{0-}$.  

Initial arcs can only preserve the multiplicity of $\Ord(K)+1$ if $q > 2p\tau$, in which case the initial arc in the rightmost copy is of type $\eta^{+-}$, ensuring that $\Ord(K)+1$ remains multiplicative.  
\end{proof}

\begin{rem}
    In the case that the maximal length arc is noninitial of type $\eta^{+-}$, the proof of Theorem~\ref{q: mult for cables} works without the bounds on $q$. Given one knows the arc is of this type beforehand, $\Ord(K)+1$ is multiplicative for all cables. This is the case if, e.g., there is a pair $(p,q)$ such that $q < 2p|\tau(K)|$, such that $\Ord(K) + 1$ is multiplicative\footnote{The immersed curve of K8a2 has eight components. All but $\gamma_0$ are figure eighths and thus arcs of type $\eta_1^{++}$ and $\eta_1^{--}$. The component $\gamma_0$ intersects $\mu$ at every integer height and thus has an initial arc of type $\eta_1^{--}$ and a noninitial arc of type $\eta_1^{+-}$. Its knot Floer order is one and it is realized by this noninitial arc.}, or if $K$ is an L-space knot (cf. Theorem~\ref{thm: Ord_v is multiplicative}). 
\end{rem}

\begin{corollary}
    Let $K$ be concordant to an L-space knot $J$ such that $\Ord(K) = \Ord(J)$. Then for all $(p,q)$-cables $K_{p,q}$, $\Ord(K) + 1$ is multiplicative in $p$. 
\end{corollary}
\begin{proof}
    Since $K$ is concordant to $J$, the immersed curve $\gamma_0$ of $K$ is equal to the immersed curve of $J$. Thus, it consists of right arcs of type $\eta_n^{+-}$ and an initial arc $\eta_1^{--}$. When $\Ord(K) = \Ord(J)$, then at least one of the maximal length arcs is of type $\eta^{+-}$. 
\end{proof}

\ifx
The cases in the remark can be collected in the following corollary. 
    Let $(\varphi)_{j \in \N}$ be the infinite family of concordance homomorphisms introduced in~\cite{DaiHomStoTru2021concordancxehomom}.

\begin{theorem}
    \[\varphi_j(K) = \#\{\alpha^+_j\} - \#\{\alpha_j^-\},\] where $\alpha^+_j$ is a downwards oriented right arc of length $j$ and $\alpha^-_j$ is an upwards oriented arc of length $j$, where $\alpha$ are the arcs coming from $\gamma_0$.
\end{theorem}
\begin{proof}
    Since we are working in $\F[U,V]/(UV = 0)$, Corollary~6.2 in~\cite{DaiHomStoTru2021concordancxehomom} proves that $\gamma_0$ is obtained from a horizontally and vertically simplified basis. As such, right arcs in $\gamma_0$ are in direct correspondence to horizontal arrows, i.e., for any arc of length $j$, there is a horizontal arrow of length $j$ in $\CFK(K)$ (cf. \cite[Remark~1.3]{DaiHomStoTru2021concordancxehomom}).
\end{proof}

\begin{corollary}
    $\Ord(K) \geq \max\{j \in \N \mid \varphi_j(K) \neq 0\}$. 
\end{corollary}
\begin{proof}
    If $\varphi_j(K) \neq 0$, there is a horizontal arrow of length $j$, thus a right arc in $\gamma_0$ of length $j$, and thus $\Ord(K) \geq j$.
\end{proof}

\begin{rem}
    This is reproves~\cite[Proposition~1.15]{DaiHomStoTru2021concordancxehomom} using immersed curves.
\end{rem}

\begin{question}
    $\Ord(K) = \max\{j \in \N \mid \varphi_j(K) \neq 0\}$
\end{question}
Probably wrong. Depends on the unidirectionality of arcs of maximal length.
\begin{question}
    Let $K$ be a knot in $S^3$ such that $\Ord(K) = \max\{j \in \N \mid \varphi_j(K) \neq 0\}$. Is then for all $(p,q)$-cables: 
\[
\Ord(K)+1 = p(\Ord(K) + 1)?
\]
\end{question}
Probably not. The existence of an arc of proper length does not imply its type.
\fi

We shortly remark the proof of Proposition~\ref{thm: explicit bounds}:

\begin{proof}
Let $\Ord(K)$ be realized by the maximal length of an arc in the immersed curve. Proposition~3.2~\cite{hom2022unknottingnumbercabling} states that for any nontrivial knot $K$, 
\[
\Ord(K_{p,q}) - 1 \geq \max\{p(\Ord(K) - 1), p - 1\}.
\]
In particular, $\Ord(K_{p,q}) \geq p(\Ord(K) - 1)$.

For the upper bound, note that the maximal length any arc can achieve in $\gamma_{p,q}$ is $pn + p - 1$ by Proposition~3.1~\cite{hom2022unknottingnumbercabling}. Since by Proposition~\ref{rem: main advancement} no longer arc can appear, $\Ord(K_{p,q}) \leq \Ord(K_{p,q}) + 1 \leq p(\Ord(K) + 1)$.  
\end{proof}

\section{Application}

We prove Theorem~\ref{thm: L-space Alex bridge} using immersed curves constructed from the Alexander polynomial. Let $\alpha_0 \geq \alpha_1 \geq \dots \geq \alpha_{2n}$ be the decreasing sequence of exponents in the Alexander polynomial $\Delta_{K_{p,q}}(t)$ of an L-space knot $K$. 

\begin{lemma}(Immersed curves for L-space knots)
\label{lem:Alex to gamma}
For an L-space knot $K$ with symmetrized Alexander polynomial $\Delta_K(t)$ with exponents $\alpha_{0} > \dots > \alpha_{2n}$, the immersed multicurve $\widetilde{\gamma}$ is a single curve constructed as follows. 
\begin{enumerate}
    \item Place a horizontal segment $[-1/4, 1/4] \times \{t\}$ at hight $t = \alpha_i$.
    \item Connect the endpoints corresponding to $\alpha_{2i+1}$ and $\alpha_{2i+2}$ on the left.
    \item Connect the endpoints corresponding to $\alpha_{2i+1}$ and $\alpha_{2i}$ on the right.
    \item Connect the unattached endpoints to $(-1/2, 0)$ and $(1/2, 0)$, respectively.
\end{enumerate}
\end{lemma}

\begin{proof}
    Following~\cite[Theorem~1.2]{OS2005lensespacesurgeries}, the knot Floer complex $\CFK(K)$ of an L-space knot $K$ is given from the Alexander polynomial $\Delta_K(t)$ with exponents $\alpha_0 \geq \dots \geq \alpha_{2n}$ by generators $\{x_i\}_{i = 0, \dots,2n}$ such that
    \begin{equation}
        \partial x_{2n} = 0 \quad\text{and}\quad \partial x_{2i + 1} = V^{d_{2i + 2}}x_{2i+2} + U^{d_{2i+1}} x_{2i},
    \end{equation}
    where $d_i = \alpha_{i-1} - \alpha_i$. Thus, $\partial x_{2i} = 0$ and $\partial x_{2i+1} = V^{d_{2i+2}} x_{2i+2}\text{ mod }U$ and therefore the basis is \textit{vertically simplified}. Also, $\partial x_{2i} = 0$ and $\partial x_{2i+1} = U^{d_{2i+1}} x_{2i} \text{ mod }V$. Thus, it is \textit{horizontally simplified} and Proposition~47 in~\cite{HanRasWat2022heegaardfortorusbdd} applies. 
\end{proof}

\begin{proof}[Proof of Theorem~\ref{thm: L-space Alex bridge}]
    Let $K$ be an L-space knot with braid index $b(K) = b$ and Seifert genus $g(K) = g$ such that 
    \[\Delta_K(t) = 1 + t - t^b - \cdots.\]
    By Equation~\eqref{eq: Ord+1 < br} and Equation~\eqref{eq: br < braid} we have
    \[
    \Ord(K) + 1 \leq \br(K) \leq b(K).
    \]
    Using Equation~\eqref{eq: ord from alex}, $\Ord(K) \geq b(K) - 1$ and thus
    \[
    \Ord(K)+1 = \br(K) = b(K).
    \]
    By Theorem~\ref{thm: Ord_v is multiplicative} and~\cite{williams1992multiplicative}, for all $(p,q)$-cables of $K$
    \[
    p\cdot (\Ord(K) + 1) = \Ord(K_{p,q}) + 1 = \br(K_{p,q}) = b(K_{p,q}) = p\cdot b(K).
    \]
    Denote by $\gamma$ the immersed curve of $K$, by $\gamma_{p,q}$ the immersed curve of $K_{p,q}$.
    The arc $a$ in $\gamma$ corresponding to $\alpha_{2g-1} - \alpha_{2g-2} = b$ is noninitial arc of maximal length of type $\eta^{+-}_{b-1}$. Thus, the arc yields an arc of type $\eta^{+-}_{p(b-1) + p - 1}$ in the rightmost copy in $\gamma_{p,q}$, which is maximal again. Since $K_{p,q}$ is an L-space knot, its curve is again monotone. Thus, two copies of $\gamma$ do not overlap during the procedure in Construction~\ref{prop: cabling process} and the arc in $\gamma_{p,q}$ corresponding to the rightmost copy of $a$ is the last nonessential arc in $\gamma_{p,q}$. Thus, its length is equal to the difference of the second-to-last and last exponent in $\Delta_{K_{p,q}}(t)$. Using~\cite{Krcatovich2018restriction}, we get
    \begin{equation*}
        \Delta_{K_{p,q}(t)} = 1 - t + t^{p(b-1) + p} - \cdots.
    \end{equation*}
    Thus, $\Ord(K_{p,q}) + 1 = p(b-1) + p = bp$, proving the result.
\end{proof}

\let\MRhref\undefined
\bibliographystyle{hamsalpha}
\bibliography{sources.bib}

\end{document}